\documentclass[12pt,a4paper]{amsart} %jesli chce numeracje jednostronnq, to daje ,oneside, w graniaste.
\usepackage{amsmath}
\usepackage{amsthm}
\usepackage{amssymb}
\usepackage{mathrsfs}
\usepackage{fancyhdr}
\usepackage{afterpage}
\newtheorem{thm}{Theorem}[section]
\newtheorem{lem}[thm]{Lemma}

\newtheorem{prop}[thm]{Proposition}

\theoremstyle{remark}
\newtheorem{rem}[thm]{Remark}
\newtheorem*{rem*}{Remark}

\theoremstyle{definition}
\newtheorem{dfn}[thm]{Definition}
\newtheorem{ex}[thm]{Example}

\numberwithin{equation}{section}

\newcommand{\om}{\Omega}
\newcommand{\ch}{{\mathcal{O}_c}}
\newcommand{\oa}{{\mathcal{O}_c^\mathrm{a}}}
\newcommand{\C}{\mathbb{C}}%{\bf C}
\newcommand{\Ker}{\operatorname{Ker}}

\begin{document}
\afterpage{\cfoot{\thepage}}
\clearpage

\title{On the Nullstellensatz for c-holomorphic functions with algebraic graphs}

\author{Adam Bia\l o\.zyt, Maciej P. Denkowski, Piotr Tworzewski}\address{Jagiellonian University, Faculty of Mathematics and Computer Science, Institute of Mathematics, \L ojasiewicza 6, 30-348 Krak\'ow, Poland}\email{adam.bialozyt@doctoral.uj.edu.pl} \email{maciej.denkowski@uj.edu.pl} \email{piotr.tworzewski@uj.edu.pl}\date{January 28th 2020}

\begin{abstract} 
C-holomorphic functions defined on algebraic sets and having algebraic graphs can be considered as a complex counterpart of \textit{regulous} functions introduced recently in real geometry. This note is a part of our study on the subject; we prove herein some effective Nullstellens\"atze.
\end{abstract}

\maketitle

\section{Introduction}

This article is a second one, after \cite{Dpreprint}, of a series devoted to continuous functions defined on a given algebraic subset of a complex finite-dimensional vector space $M$ and having algebraic graphs, i.e. the class of c-holomorphic functions with algebraic graphs. They can be seen as a kind of complex counterpart of the recently introduced {\it regulous} functions \cite{FHMM}, \cite{Kol} in connection with \cite{Krz}. 

To simplify the notation, we will consider $M={\C}^m$ (anyway, everything here is invariant under linear isomorphisms). In this paper we study some basic Nullstellens\"atze in this class. Dealing with such functions requires the use of purely geometric methods.

For the convenience of the reader we recall some basic facts. Let $A\subset\om$ be an analytic subset of an open set $\om\subset{\C}^m$. R. Remmert generalized the notion of holomorphic mapping onto
sets having singularities in a more convenient way (from the geometric point of view) than the usual notion of {\it weakly holomorphic functions} (i.e. functions defined and holomorphic on $\mathrm{Reg} A$ and locally bounded on $A$), namely:

%\smallskip
\begin{dfn} (cf. \cite{R}) A mapping $f\colon
A\to{\C}^n$ is called {\it c-ho\-lo\-morph\-ic} if it is
continuous and the restriction of $f$ to the subset of regular
points $\mathrm{Reg} A$ is holomorphic. We denote by $\ch (A,{\C}^n)$
the vector space of c-holomorphic mappings, and by $\ch (A)$ the ring of
c-holomorphic functions. 
\end{dfn}
%\smallskip

The corner stone of all our further considerations is the following theorem:

%\smallskip
\begin{thm}[\cite{Wh} 4.5Q] A mapping $f\colon A\to {\C}^n$ defined on an analytic set $A\subset{\C}^m$ 
is c-holomorphic iff it is continuous and its graph
$\Gamma_f:=\{(x,f(x))\mid x\in A\}$ is an analytic subset of
$\om\times{\C}^n$.\end{thm}

%\smallskip
For a more detailed list of basic properties of c-holomorphic mappings see \cite{Wh}, \cite{D} and \cite{Dpreprint}. It is easy to see that for $f\in{\ch}(A,{\C}^n)$, the graph $\Gamma_f$ is irreducible iff $A$ is irreducible.

Finally, we recall some notions we will be using. %For a polynomial $P\in\mathbb{C}[x_1,\dots, x_m]$ let $P^+$ denote its homogeneous part of maximal degree, i.e. $\deg P^+=\deg P$ and $\deg (P-P^+)<\deg P$. We denote by $\tilde{P}(t,z)=\sum_{|\alpha|\leq d}a_{\alpha} z^\alpha t^{d-|\alpha|}$ the homogenization of $P(z)=\sum_{|\alpha|\leq d} a_{\alpha} z^\alpha$ where $d=\deg P$. Then, $\tilde{P}(0,z)=P^+(z)$. 
If $\Gamma\subset{\C}^n$ is algebraic of pure dimension $k$, then $\deg \Gamma=\#(L\cap \Gamma)$ for any $L\subset{\C}^n$ affine subspace of dimension $n-k$ transversal to $\Gamma$ and such that $L_\infty\cap\overline{\Gamma}=\varnothing$, where $\overline{\Gamma}$%=\overline{\PP(\{1\}\times \Gamma)}$, for $\PP\colon{\C}^{1+n}_*\to\PP_n$ the canonical projection, 
is the projective closure and $L_\infty$ denotes the points of $L$ at infinity (see \cite{L} VII.\S 7 and \cite{Ch}).% i.e. the intersection of $\overline{L}$ with the hyperplane at infinity in $\mathbb{P}_n$. 
The condition  $L_\infty\cap\overline{\Gamma}=\varnothing$ is equivalent both to $L_\infty\cap{\Gamma}_\infty=\varnothing$) and to the inclusion
$$
\Gamma\subset\{u+v\in L'+L\mid ||v||\leq\mathrm{const.} (1+||u||)\}
$$
where $L'$ is any $k$-dimensional affine subspace such that $L'+L={\C}^n$.
Moreover, for any $(n-k)$-dimensional affine subspace $L$ intersecting $A$ in a zero-dimensional set, $\#(L\cap\Gamma)\leq\deg \Gamma$. 

We write $G'_{n-k}({\C}^n)$ for the set of affine hyperplanes of dimension $n-k$ and we have
$$
\deg \Gamma=\max\{\#(L\cap \Gamma)\mid L\in G'_{n-k}({\C}^n)\colon \dim (L\cap \Gamma)=0\}.
$$ 
%The projective degree of $\Gamma$ is in fact equal to the local degree (Lelong number) at zero of the cone defined by $\Gamma$ (\footnote{Namely, if $C_\Gamma$ denotes the closure in ${\C}\times{\C}^n$ of the pointed cone $S_\Gamma:=\{{\C}_*\cdot (1,x)\mid x\in \Gamma\}$, then $\deg \Gamma=\deg_0 C_\Gamma$. Of course, $C_\Gamma=\PP^{-1}(\overline{\Gamma})\cup\{0\}$}).
%
%We write $\Gamma^*:=\{x\in{\C}^n\mid (0,x)\in C_\Gamma\}$ for the set in affine space corresponding to the points at infinity. In particular, $\{P^+=0\}=\{P=0\}^*$. %Moreover, $\Gamma^*=\{v\in{\C}^n\mid \exists \Gamma\ni v_\nu\to \infty, \lambda_\nu\in{\C}\colon \lambda_\nu v_\nu\to v\}$ and this is a cone. Finally, observe that 
%Now $V_\infty\cap W_\infty=\varnothing$, where $V,W\subset{\C}^n$ are algebraic, is equivalent to $V^*\cap W^*=\{0\}$.

%Indeed, if $V=\bigcap_1^k P_j^{-1}(0)$, then $C_V=\bigcap_1^k \tilde{P}^{-1}(0)$ (and this can be considered as a subset of $\PP_m$ since the forms are homogeneous). This follows from the fact that apart from $\{t=0\}$, $\tilde{P}_j(t,z)=t^d_j\tilde{P}_j(z/t)$ and so $S_V=C_V\setminus(C_V\cap H)$, where $H=0\times{\C}^m$. As a difference of two analytic sets, $\overline{\PP(S_V)}=\PP(C_V\setminus 0)$ is analytic when seen in $\PP_m$ and coincides with $\overline{V}$. Then $\PP^{-1}(\overline{V})=C_V$. 

%\medskip

A natural subclass of c-holomorphic functions that we can expect will behave like polynomials or regular functions is formed by those c-holomorphic function whose graph is algebraic. 

\begin{dfn}
We will call {\it c-algebraic} any continuous function $f\colon A\to{\C}$ having an algebraic graph, where $A\subset{\C}^m$ is a fixed analytic set. 
We will denote by $$\mathcal{O}_c^\mathrm{a}(A)=\{f\in{\ch}(A)\mid \Gamma_f\> \textrm{is algebraic}\}$$ the ring of such functions and write $f=(f_1,\dots, f_n)\in{\oa}(A,{\C}^n)$ whenever all the components $f_j$ are c-algebraic which is clearly equivalent to $f$ having an algebraic graph.

\end{dfn}
\begin{rem} Of course, by the Chevalley-Remmert Theorem 
$$
{\oa}(A)\neq\varnothing\>\Rightarrow\> A\>\textrm{is algebraic}.
$$
\end{rem}
\begin{ex}
The eternal example of a c-algebraic but non-regular function is $f(x,y)=y/x$ on $A\colon y^2=x^3$ extended at the origin of ${\C}^2$ by putting $f(0,0)=0$. 
\end{ex}
As observed in \cite{Dpreprint} each c-algebraic function on an algebraic pure dimensional set is the restriction of a rational function (this is the link to \cite{FHMM}).

From now on we will assume that $A\subset{\C}^m$ is a pure $k$-dimensional algebraic set of degree $d:=\deg A$. Obviously, we shall assume also $k\geq 1$ unless something else is stated.

Let $||\cdot||$ denote one of the usual norms in ${\C}^m$. In \cite{Dpreprint} we studied the {\it growth exponent} at infinity that generalizes the notion of degree of a polynomial for $f\in\mathcal{O}_c^\mathrm{a} (A)$:
$$
\mathcal{B}(f):=\inf\{s\geq 0\mid |f(x)|\leq \mathrm{const.}||x||^s,\ x\in A\colon ||x||\gg 1\}
$$
where the defining inequality is equivalent to 
$$
|f(x)|\leq \mathrm{const.}(1+||x||)^s,\quad x\in A\leqno{(*)}
$$
and this characterises ${\oa}(A)$ in ${\ch}(A)$, cf. \cite{Dpreprint} Lemma 3.1. In the polynomial case, i.e. $A={\C}^k$, we have $\mathcal{B}(f)=\deg f$, for c-algebraic functions are polynomials by the Serre Graph Theorem.

%The defining inequalities are, of course, equivalent, and instead of $(1+|x|)^s$ we may as well write $(1+|x|^s)$. 

\section{Characteristic polynomials}

 In the setting of c-algebraic functions we are obliged to make do more with the geometric structure than the algebraic one which is a basic hindrance when trying to transpose the results known, for instance, for polynomial dominating mappings to the c-holomorphic algebraic case.

We consider now the following situation:\\ Let $A\subset{\C}^m$ be algebraic of pure dimension $k>0$ and suppose $f\in{\oa}(A,{\C}^k)$ is a proper mapping. %Then each component $f_j$ of $f$ has an algebraic graph. 

Since $\Gamma_f$ is algebraic with proper projection onto ${\C}^k$, the cardinality of the fibre $\# f^{-1}(y)$ is constant for the generic point $y\in{\C}^k$. Following \cite{Dpreprint} we denote this number by $\mathrm{d}(f)$ and call it the {\it geometric degree} of $f$ just as in the polynomial case. We call {\it critical} for $f$ any point $y\in{\C}^k$ for which $\#f^{-1}(w)\neq\mathrm{d}(f)$. In that case one has actually $\#f^{-1}(w)<\mathrm{d}(f)$.% (cf. e.g. \cite{Ch}, the projection onto ${\C}^k$ restricted to $\Gamma_f$ is a $\mathrm{d}(f)$-sheeted branched covering). 
Obviously $\mathrm{d}(f)\leq\deg\Gamma_f$ (cf. \cite{L}). 
 
%By \cite{Dpreprint} we have a Bezout-type inequality:
%\begin{prop}[\cite{Dpreprint} Proposition 5.1]\label{Bezout}
%Let $f\colon A\to{\C}^k$ be a c-holomorphic proper mapping with algebraic graph. Then 
%$$
%\mathrm{d}(f)\leq\deg A\prod_{j=1}^k\mathcal{B}(f_j).
%$$
%\end{prop}

In the sequel we shall use intensively the notion of characteristic polynomial relative to a proper mapping $f$.
For any $g\in\mathcal{O}_c^\mathrm{a}(A)$ let us introduce the {\it characteristic polynomial} of $g$ relative to $f$: for any $y\in{\C}^k$ not critical for $f$ we put
$$
P_g(y,t):=\prod_{x\in f^{-1}(y)}(t-g(x))=t^{\mathrm{d}(f)}+a_1(y)t^{\mathrm{d}(f)-1}+\ldots+a_{\mathrm{d}(f)}(y)\leqno{(\#)}
$$
extending the coefficients through the critical locus of $f$ thanks to the Riemann Extension Theorem (they are continuous; see below their form). Therefore $P_g\in\mathcal{O}({\C}^k)[t]$. 
\begin{prop}\label{wman}
In the introduced setting, $P_g$ is a pure-bred polynomial, i.e. $P_g\in{\C}[y_1,\dots,y_k][t]$.
\end{prop}
\begin{proof}
This follows directly from the expressions for the coefficients:
$$
a_j(y)=(-1)^j\sum_{1\leq \iota_1<\ldots<\iota_j\leq \mathrm{d}(f)}g(x^{(\iota_1)})\cdot\ldots\cdot g(x^{(\iota_j)}),
$$
where $f^{-1}(y)=\{x^{(1)},\dots, x^{(\mathrm{d}(f))}\}$ consists of $\mathrm{d}(f)$ points. 

Since $g\in\mathcal{O}_c^\mathrm{a}(A)$, there is $|g(x)|\leq C_1(1+||x||^{r})$ for $x\in A$ with some constants $C_1,r>0$ (cf. $(*)$  and \cite{Dpreprint} Lemma 3.1). By assumption, $\Gamma_f$ has proper projection onto ${\C}^k$ and so by the Rudin-Sadullaev Theorem (see \cite{L}),
$$
\Gamma_f\subset\{(x,y)\in{\C}^m\times{\C}^k\mid ||x||\leq C_2(1+||y||)^s\}
$$
for some constants $C_2,s>0$. Therefore, for any $x\in A$, $||x||\leq C_2(1+||f(x)||)^s$. We obtain thus
$$
|g(x)|\leq C_1(1+C_2^r(1+||f(x)||)^{rs})\leq C_1 2\max\{1,C_2^r\}(1+||f(x)||)^{rs}.
$$
This means in particular that for any $y$ not critical for $f$, and for all $j$, 
$$
|a_j(y)|\leq\mathrm{const.} (1+||y||)^{rsj},
$$
since $y=f(x^{(j)})$. By continuity this inequality can be extended to the whole of ${\C}^k$ and so by Liouville's Theorem $a_j\in{\C}[y_1,\dots, y_k]$ for all $j$. 
\end{proof}
\begin{rem}\label{uwaga}
In the proof above we may put $r=\mathcal{B}(g)$. On the other hand, by \cite{TW1}, we can also take $s=\deg\Gamma_f-\mathrm{d}(f)+1$. Therefore, we obtain $$\deg a_j\leq j\mathcal{B}(g)(\deg\Gamma_f-\mathrm{d}(f)+1).$$
\end{rem}

The following Lemma of P\l oski is crucial when using the characteristic polynomial:% (especially for the study of  the \L ojasiewicz exponent, see \cite{BDT}):
\begin{lem}[\cite{P} lemma (2.1)]\label{P} Let $P(x,t)=t^d+a_1(x)t^{d-1}+\ldots+a_d(x)$ be a polynomial with $a_j\in{\C}[x_1,\dots,x_k]$. Then $\delta(P):=\max_{j=1}^d ({\deg a_j}/{j})$ is the minimal exponent $q>0$ such that 
$$
\{(x,t)\in{\C}^k\times{\C}\mid P(x,t)=0, |x|\geq R\}\subset\{(x,t)\in{\C}^k\times{\C}\mid |t|\leq C|x|^q\}
$$
for some $R, C>0$. 
\end{lem}

We present a simple application of the characteristic polynomial in the polynomial case. Let $f\colon {\C}^m\to{\C}$ be a polynomial of degree $d$ and such that the complex gradient (with respect to $\partial/\partial z_j$) $\nabla f\colon{\C}^m\to{\C}^m$ is proper with multiplicity $\mu$ (\footnote{Which is equivalent to  saying that the fibres are finite and the sum of the local multiplicities always gives $\mu$.}). Let $P(y,t)$ be the characteristic polynomial of $f$ with respect to $\nabla f$. By Proposition \ref{wman} it is a polynomial $P\in{\C}[y_1,\dots,y_m,t]$ of the form
$$
P(y,t)=\prod_{x\in \nabla f^{-1}(y)}(t-f(x))=t^{\mu}+a_1(y)t^{\mu-1}+\ldots+a_{\mu}(y).
$$

Let us write $D=\deg\Gamma_{\nabla f}$. Obviously, $D\geq \mu$.
\begin{prop}
In the setting introduced above, 
$$
|f(x)|^\Theta\leq \mathrm{const.}||\nabla f(x)||,\quad ||x||\gg 1
$$
where $$\Theta=\frac{1}{d(D-\mu+1)}\in \left(0,\frac{1}{d}\right].$$
\end{prop}
\begin{proof}
This follows essentially from Lemma \label{P} and Remark \ref{uwaga}, but we may also give a direct argument. Clearly, $f(x)$ is a root of $P(\nabla f(x),\cdot)$. Therefore, 
$$
||f(x)||\leq 2\max_{j=1}^\mu |a_j(\nabla f(x))|^\frac{1}{j}.
$$
By Remark \ref{uwaga} we have 
$$
|a_j(\nabla f(x))|^\frac{1}{j}\leq \mathrm{const}\cdot ||\nabla f(x)||^{d(D-\mu+1)}
$$
whence the result.
\end{proof}

\section{Nullstellens\"atze and degree of cycles of zeroes}
We shall deal first with the $0$-dimensional case, i.e. we assume that $f=(f_1,\dots, f_n)$ is a proper c-algebraic map on a set $A\subset{\C}^m$ of pure dimension $k>0$ (it follows that $n\geq k$). % as in the preceding section. It is clear that it is surjective. With  
%all the notations introduced so far %and the non-restrivtive additional assumption that $0\in A$ and $f(0)=0$, 
We assume that $f^{-1}(0)\neq\varnothing$. Then the intersection of the graph $\Gamma_f$ with ${\C}^m\times\{0\}^n$ is isolated but possibly improper (\footnote{Recall that the intersection $X\cap Y$ of two pure-dimensional analytic sets $X,Y$ in a manifold $M$ is called proper, if at any $a\in X\cap Y$ it has the minimal possible dimension $\dim X+\dim Y-\dim M$.}). We may compute the improper intersection indices $$i_a(f):=i(\Gamma_f\cdot({\C}^m\times\{0\}^n);a),\quad a\in f^{-1}(0)$$
along \cite{ATW} (see also \cite{Dm}). 

The properness of $f$ implies by the Chevalley-Remmert Theorem that $f(A)$ is algebraic. If it happens to be {\sl irreducible}, then we can define the {\it generalized geometric degree} $\mathrm{d}(f)$ to be the degree (sheet number) of the branched covering (cf. \cite{Ch}) $f|_{A\setminus f^{-1}(\mathrm{Sng}f(A))}$ over the connected manifold $\mathrm{Reg}f(A)$. %We will see that $\mathrm{d}(f)\deg f(A)\geq \sum_{a\in f^{-1}(0)}i_a(f)$. 

\begin{ex}
It should be pointed out that although $\mathrm{d}(f)$ correpsonds to the cardinality of the generic nonempty fibre of $f$, it may not be the greatest cardinality of the fibre: if $A$ denotes the graph of $t\mapsto (t^2-1,t(t^2-1))$ in ${\C}^3$ and $f$ is the projection onto the last two coordinates, then $\mathrm{d}(f)=1$, whereas $\#f^{-1}(0)=2$.
\end{ex}

For $n=k$, which means that the intersection is proper and $f$ is surjective, the numbers $i_a(f)$ coincide with the usual local geometric multiplicities $m_a(f)$ (\footnote{I.e. $m_a(f)=\limsup_{v\to f(a)}\#(U\cap f^{-1}(v))$ where $U$ is a neighbourhood of $a$ such that $\overline{U}\cap f^{-1}(f(a))=\{a\}$ (the result is the same for all $U$ small enough); see \cite{Ch}.}) and the Stoll formula yields $\mathrm{d}(f)=\sum_{a\in f^{-1}(0)}m_a(f)$. Thanks to a Spodzieja-type reduction as in \cite{Sp},  
we obtain the following general c-algebraic counterpart of a result of E. Cygan \cite{Cg}: 

\begin{thm}\label{Nullst}
In the setting introduced above, let $g\in\mathcal{O}_c^\mathrm{a}(A)$ be such that $g^{-1}(0)\supset f^{-1}(0)$. Then, if $f(A)$ is irreducible, there are $n$ functions $h_j\in\mathcal{O}_c^\mathrm{a}(A)$ such that 
$$
g^{\mathrm{d}(f)\deg f(A)}=\sum_{j=1}^n h_j f_j\quad\hbox{\it on the whole of}\quad A.
$$
Moreover, $$\mathrm{d}(f)\deg f(A)\geq \sum_{a\in f^{-1}(0)}i_a(f)$$ and equality holds, provided $k=n$ (in which case $\deg f(A)=1$).
\end{thm}
In the proof we will need the following refinement of \cite{Dm} Theorem 2.3:
\begin{prop}\label{m}
Let $A\subset{\C}^m\times D$ be a pure $k$-dimensional analytic set with proper projection $p|_A$ where $p\colon {\C}^m\times D\ni (x,y)\mapsto y\in D$ and $D\subset{\C}^n$ is a domain. Assume that $X:=p(A)$ is irreducible, $0\in X$ and let $\mu$ denote the multiplicity of the branched covering $p|_{A\setminus p^{-1}(\mathrm{Sng} X)}$ over the connected manifold $\mathrm{Reg} X$. Then 
$$
\mu\cdot \deg_0 X=\sum_{a\in p^{-1}(0)\cap A}i(A\cdot\Ker p; a),
$$
where $\deg_0 X$ stands for the local degree (Lelong number) of $X$ at the origin.
\end{prop}
\begin{rem}
Asearlier we are dealing here with a possibly improper isolated intersection $A\cap \Ker p$ and we are using \cite{ATW} to compute the intersection multiplicities which we may denote as earlier $i_a(p|_A):=i(A\cdot\Ker p; a)$.
\end{rem}
\begin{proof}[Proof of Proposition \ref{m}] 
We can find a neighbourhood $U$ of $0\in{\C}^n$ and pairwise disjoint neighbourhood $V_1,\dots, V_\nu$ ($\nu\leq \mu$) of the points in the fibre $p^{-1}(0)\cap A$ such that $U\cap X=\bigcup_{j=1}^r X_j$ corresponds to the decomposition of the germ $(X,0)$ into irreducible components and $p^{-1}(U)=\bigcup_{i=1}^\nu V_i\cap A$. We regroup the sets $$W_j:=\bigcup_{s=1}^{t_j}V_{i_s}\cap A$$ according to their projection, i.e. to obtain $W_j$ we take all the sets $V_{i_s}\cap A$, $s=1,\dots,t_j$, for which $p(V_i\cap A)=X_j$. We get as much sets $W_j$ as there are components $X_j$ and $p|_{W_j}\colon W_j\to X_j$ is proper and becomes a $\mu$-sheeted branched covering over $\mathrm{Reg} X_j$. 

Moreover, the points of $p^{-1}(0)\cap A$ form groups $\{a_{i_1},\dots,a_{i_{t_j}}\}=p^{-1}(0)\cap W_j$ and we may compute the {\it regular multiplicity} (cf.\cite{Dm}) $$\tilde{m}_{a_{i_s}}(p|_A)=\limsup_{\mathrm{Reg}p(X)\ni y\to 0}\#p^{-1}(y)\cap A\cap V_{i_s}=\tilde{m}_{a_{i_s}}(p|_{W_j}),$$ provided $V_{i_s}$ is small enough (which, of course, we may assume with no harm forthe generality). Then, clearly, $\mu=\sum_{s=1}^{t_j}\tilde{m}_{a_{i_s}}(p|_A)$. Finally, by \cite{Dm} Theorem 2.3,
$$
\tilde{m}_{a_{i_s}}(p|_{W_j})\cdot\deg_0 X_j=i_a(p|_{W_j})
$$
and $i_{a_{i_s}}(p|_{W_j})=i_{a_{i_s}}(p|_A)$. Therefore,
\begin{align*}
\sum_{i=1}^\nu i_{a_i}(p|_A)&=\sum_{j=1}^r\sum_{s=1}^{t_j}\left(\tilde{m}_{a_{i_s}}(p|_{W_j})\cdot\deg_0 X_j\right)=\\
&=\sum_{j=1}^r \mu\cdot \deg_0 X_j=\mu\cdot\deg_0 X
\end{align*}
as required.
\end{proof}

\begin{proof}[Proof of Theorem \ref{Nullst}]
First, let us deal with the proper intersection case, i.e. let us assume that $k=n$. Let $P_g$ be the characteristic polynomial of $g$ relative to $f$ as in $(\#)$. From the definition we have clearly $P_g(f(x),g(x))=0$ for $x\in A$, which means 
$$
g(x)^{\mathrm{d}(f)}=-a_1(f(x))g(x)^{\mathrm{d}(f)-1}+\ldots-a_{\mathrm{d}(f)}(f(x)).
$$
Now, for any $j$, $a_j\in{\C}[y_1,\dots, y_k]$ and since $g=0$ on $f^{-1}(0)$, it follows from the expression of $a_j$ (see the proof of Proposition \ref{wman}) that $a_j(0)=0$ for any $j$. Therefore $a_j(y)=\sum_{\iota=1}^k a_{j,\iota}(y)y_\iota$ with $a_{j,\iota}\in{\C}[y_1,\dots, y_k]$ and the assertion follows.

Now, assume that $k>n$ and write $X:=f(A)$. Then $X$ is an irreducible $k$-dimensional algebraic subset of ${\C}^n$ and for the generic epimorphism $\pi\colon{\C}^n\to{\C}^k$ we have $\mathrm{d}(\pi|_X)=\deg X$. Consider the proper c-algebraic map $\tilde{f}:=\pi\circ f$. From the first part of the proof we get 
$$
g^{\mathrm{d}(\tilde{f})}=\sum_{j=1}^k h_j\tilde{f}_j,
$$
but $\tilde{f}_j=(\pi\circ f)_j=(\pi_j\circ f)$ and $\pi_j(y)=\sum_{i=1}^n\alpha_{ji}y_i$ which yields the assertion sought after, provided we show $\mathrm{d}(\tilde{f})=\mathrm{d}(f)\cdot \deg X$. But for the generic $x\in{\C}^k$, the fibre $\pi^{-1}(x)$ lies entirely in $\mathrm{Reg} A$ and does not meet the critical set of the $\mathrm{d}(f)$-sheeted branched covering $f|_{A\setminus f^{-1}(\mathrm{Sng}X)}$ over the connected manifold $\mathrm{Reg} X$, whence it follows that $\mathrm{d}(\tilde{f})=\mathrm{d}(f)\cdot \deg X$. 

To prove the second part of the assertion we may assume that the coordinates in ${\C}^n$ are chosen so that $\pi$ is the projection onto the first $k$ of them. Put $T:={\C}^m\times\{0\}^k$ and take the linear projection $\tau$ such that $\Ker \tau=T$. % then $N=\Ker p$, for $p(x,u,v)=u$, $(x,u,v)\in{\C}^m\times{\C}^k\times{\C}^{n-k}$ and so 
Then, by the Stoll Formula
$$
\mathrm{d}(\tilde{f})=\mathrm{d}(\tau|_{\Gamma_{\tilde{f}}})=\sum_{a\in T\cap \Gamma_{\tilde{f}}}i(T\cdot\Gamma_{\tilde{f}};a)
$$
where now the intersection multiplicites $i(T\cdot\Gamma_{\tilde{f}};a)$ are computed along Draper, as we are dealing with an isolated proper intersection. In particular $\mathrm{d}(\tau|_{\Gamma_{\tilde{f}}})\geq \sum_{b\in f^{-1}(0)} i(T\cdot\Gamma_{\tilde{f}};(b,0))$ where $0\in{\C}^k$. 

Observe that $(\mathrm{id}_{{\C}^m}\times\pi)|_{\Gamma_f}=\Gamma_{\tilde{f}}$ and this is a one to one correspondence. Put $N:={\C}^m\times\{0\}^k\times{\C}^{n-k}$. It is classical in proper intersection theory that $i(\Gamma_{\tilde{f}}\cdot T;(b,0))=m_{(b,0)}(\tau|_{\Gamma_{\tilde{f}}})$. Similarly, $i(\Gamma_f\cdot N;(b,0))=m_{(b,0)}(p|_{\Gamma_f})$ (here $0\in{\C}^n$) for the linear projection $p$ with $\Ker p=N$. Now, if $z\in{\C}^k$, then the fibres $\tau^{-1}(z)\cap\Gamma_{\tilde{f}}\cap (U\times V)$ and $p^{-1}(z)\cap\Gamma_f\cap(U\times V\times W)$ are bijective by $\mathrm{id}_{{\C}^m}\times \pi$, where $b\in U\subset{\C}^m$, $0\in V\subset{\C}^k$, $0\in W\subset{\C}^{n-k}$ are sufficiently small neighbourhoods. This implies that $m_{(b,0)}(\tau|_{\Gamma_{\tilde{f}}})=m_{(b,0)}(p|_{\Gamma_f})$ and we are done.
\end{proof}
\begin{rem}
Example 3.3 from \cite{D2} shows that the coefficients $h_j$ may well be strictly c-algebraic, i.e. having no holomorphic extension onto a neighbourhood of $A$ in ${\C}^m$ (even locally).
\end{rem}

Theorem \ref{Nullst} can be further generalized in two directions. The first one is the c-algebraic counterpart of Lemma 1.1 from \cite{PT} (compare also \cite{D} for the c-holomorphic version which was much more straightforward) and its important consequence.
\begin{prop}\label{PT}
Let $f=(f_1,\dots,f_k)\in{\oa}(A,{\C}^k)$ be proper on the pure $k$-dimensional set $A\subset{\C}^m$, with geometric degree $\mathrm{d}(f)$. Fix $\ell\in\{1,\dots,k\}$ and assume that $g\in{\oa}(A)$ satisfies $g^{-1}(0)\supset \bigcap_{j=1}^\ell f^{-1}_j(0)$. Then we can find $\ell$ functions $h_j\in{\oa}(A)$ for which
$$
g^{\mathrm{d}(f)}=\sum_{j=1}^\ell h_j f_j,\quad \textit{on}\> A.
$$
\end{prop}
\begin{proof}
By the Chevalley-Remmert Theorem, the set $$X:=\{(f(x),g(x))\in \mathbb{C}^k\times\mathbb{C}\mid x\in A\}$$ is algebraic as the proper (\footnote{This is a simple topological fact: given three locally compact topological spaces $X,Y,Z$ and two continuous functions $\varphi\colon X\to Y$ and $\psi\colon X\to Z$ such that $\psi$ is in addition proper, the function $(\varphi,\psi)\colon X\to Y\times Z$ is proper.}) projection of the graph of $(f,g)$ and it has pure dimension $k$. It is obvious that the characteristic polynomial $P_g$ (having the form $(\#)$) of $g$ relative to $f$ yields $X=P_g^{-1}(0)$. Moreover, 
$$
X\cap (\{0\}^\ell\times\mathbb{C}^{k-\ell}\times\mathbb{C})=\{0\}^\ell\times \mathbb{C}^{k-\ell}\times\{0\}\ \Rightarrow\ a_j|_{\{0\}^\ell\times \mathbb{C}^{k-\ell}}\equiv 0,
$$
for all indices $j$ (which is easy to check starting from the last one, i.e. $a_d$). This allows us to write 
$a_j(w)=y_1a_{j,1}(w)+\ldots+y_\ell a_{j,\ell}(w)$ where $w=(y,z)\in \mathbb{C}^\ell\times\mathbb{C}^{k-\ell}$ and $a_{j,s}$ are polynomials. Eventually, the identity $P_g(f(x),g(x))\equiv 0$ gives the result.
\end{proof}

We are able now to generalize this to the case of a set-theoretical complete intersection for strictly regular sequences in connection with \cite{FPT}, \cite{D2} and \cite{PT}. Suppose that $f\colon A\to{\C}^n$ is c-algebraic, $A$ has pure dimension $k>0$ and $f^{-1}(0)$ is pure $(k-n)$-dimensional. This means exactly that the intersection $\Gamma_f\cap({\C}^m\times\{0\}^k)$ is proper (i.e. it is a set-theoretical complete intersection). In such a case we may consider the {\it algebraic effective cycle of zeroes} of $f$:
$$
Z_f:=\Gamma_f\cdot({\C}^m\times\{0\}^n)=\sum_{j=1}^r i(\Gamma_f\cdot({\C}^m\times\{0\}^n);V_j)V_j,\leqno{(\star)}
$$
where $f^{-1}(0)=\bigcup_{j=1}^r V_j$ is the decomposition into irreducible components and $i(\Gamma_f\cdot({\C}^m\times\{0\}^n);V_j)$ is the intersection multiplicity along $V_j$ computed following Draper \cite{Dr} (see also \cite{Ch}).

Since all $V_j$ are algebraic we may define the {\it degree of the cycle} $Z_f$ to be the number
$$
\deg Z_f=\sum_{j=1}^r i(\Gamma_f\cdot({\C}^m\times\{0\}^n);V_j)\cdot\deg V_j.
$$
Note that for $k=n$ we clearly obtain $\deg Z_f=\mathrm{d}(f)$.

Following \cite{FPT} we introduce the notion of a strictly regular sequence:
\begin{dfn}
A map $f\in{\oa}(A,{\C}^n)$ where $A\subset{\C}^m$ has pure dimension $k>n$, is called {\it strictly regular}, if there are $k-n$ affine forms $L_j\colon{\C}^m\to{\C}$ such that $(f,L_1|_A,\dots,L_{k-n}|_A)\colon A\to{\C}^k$ is proper.
\end{dfn}
\begin{rem}
If $f\in{\oa}(A,{\C}^n)$ is strictly regular, then $f^{-1}(0)$ has pure dimension $k-n$ (cf. \cite{D2}) and so $Z_f$ is well-defined.
\end{rem}

\begin{thm}\label{deg}
Assume $f\in{\oa}(A,{\C}^n)$ is strictly regular. Then for any $g\in\mathcal{O}_c^\mathrm{a}(A)$ such that $g^{-1}(0)\supset f^{-1}(0)$ there are $n$ functions $h_j\in\mathcal{O}_c^\mathrm{a}(A)$ yielding
$$
g^{\deg Z_f}=\sum_{j=1}^n h_jf_j\quad\hbox{\it on the whole of}\ A.
$$
\end{thm}
\begin{proof}
Let $L=(L_1,\dots, L_{k-n})$ be the affine mapping from the definition of strict regularity. Then $\varphi:=(f,L|_A)\in{\oa}(A,{\C}^k)$ is a proper mapping. We may assume that coordinates in ${\C}^m$ are chosen in such a way that $L$ is linear. Observe that $\varphi^{-1}(u,v)=f^{-1}(u)\cap L^{-1}(v)$, for $(u,v)\in{\C}^n\times{\C}^{k-n}$. Consider $\psi(\zeta,u):=(u,L(\zeta))\in{\C}^n\times{\C}^{k-n}$, for $(\zeta,u)\in{\C}^m\times{\C}^n$. It is easy to see that $\psi|_{\Gamma_f}$ is proper: if $K\subset{\C}^k$ is a compat set, then $\psi^{-1}(K)\cap \Gamma_f=\{(\zeta,f(\zeta))\mid \zeta\in\varphi^{-1}(K)\}$ is clearly compact.

Then $\mathrm{d}(\varphi)=\mathrm{d}(\psi|_{\Gamma_f})$, because $(\psi|_{\Gamma_f})^{-1}(u,v)=\{(\zeta, f(\zeta))\mid \zeta\in f^{-1}(u)\cap L^{-1}(v)\}$. The intersection $\Gamma_f\cap \Ker \psi$ is isolated and proper and, obviously, $\mathrm{d}(\psi|_{\Gamma_f})=\deg(\Gamma_f\cdot\Ker\psi)$. Since $\Ker\psi\subset{\C}^m\times\{0\}^n$, we can apply \cite{TW2} Theorem 2.2:
\begin{align*}
\Gamma_f\cdot_{{\C}^{m+n}}\Ker\psi&=(\Gamma_f\cdot_{{\C}^{m+n}}({\C}^m\times\{0\}^n))\cdot_{{\C}^m}\Ker\psi=\\
&=Z_f\cdot\Ker L.
\end{align*}
Therefore, in view of $(\star)$, $$\mathrm{d}(\varphi)=\deg(Z_f\cdot\Ker L)=\sum_{j=1}^ri(\Gamma_f\cdot({\C}^m\times\{0\}^n); V_j)\deg(V_j\cdot \Ker L),$$
where $\deg(V_j\cdot\Ker L)=\sum_{a\in V_j\cap\Ker L}m_a(\pi^L|_{V_j})$ with the local geometric multiplicities $m_a(\pi^L|_{V_j})$ of the projection $\pi^L$ along $\Ker L$. Note that we have necessarily $\deg(V_j\cdot\Ker L)\leq \deg V_j$ whence $\mathrm{d}(\varphi)\leq \deg Z_f$. Now, Proposition \ref{PT} gives 
$$
g^{\mathrm{d}(\varphi)}=\sum_{j=1}^n g_j f_j,\quad \textrm{on}\> A,
$$
with some $g_j\in{\oa}(A)$. Multiplying both sides by $g^{\deg Z_f-\mathrm{d}(\varphi)}$ and putting $h_j:=g^{\deg Z_f-\mathrm{d}(\varphi)}\cdot g_j$ we get the result sought for.
\end{proof}

\end{document}